\documentclass[letterpaper, 10 pt, conference]{ieeeconf}  

\IEEEoverridecommandlockouts                              

\usepackage{amssymb}  

\usepackage{graphicx}
\usepackage{url}
\usepackage{algorithm}
\usepackage{algorithmic}

\newtheorem{theorem}{Theorem}
\newtheorem{definition}[theorem]{Definition}
\newtheorem{lemma}[theorem]{Lemma}

\newcommand{\R}{\mathbb{R}}
\newcommand{\Z}{\mathbb{Z}}

\title{\LARGE \bf
On the Inapproximability of the Discrete Witsenhausen Problem 
}

\author{Alex Olshevsky$^{1}$
\thanks{$^{1}$Alex Olshevsky is with the Department of Electrical and Computer Engineering and the Division of Systems Engineering, Boston University, Boston, USA, 
        {\tt\small alexols@bu.edu}%
}}

\begin{document}

\maketitle
\thispagestyle{empty}
\pagestyle{empty}

\begin{abstract} We consider a discrete version of the Witsenhausen problem where all random variables are bounded and take on integer values. Our main goal is to 
understand the complexity of computing good strategies given the distributions for the initial state and second-stage noise as inputs to the problem. Following Papadimitriou 
and Tsitsiklis \cite{papadimitriou1986intractable}, who showed that computing the optimal solution  is NP-complete, we construct a sequence of problem instances with the initial state uniform over
a set of size $n$ and the noise  uniform over a set of size at most $n^2$, such that finding a strategy whose cost is a multiplicative $n^{2-\epsilon}$ approximation to the optimal cost is NP-hard for any $\epsilon > 0$.

\end{abstract}

\section{Introduction} Witsenhausens's seminal counterexample \cite{witsenhausen1968counterexample} demonstrated that linear strategies are not always always in sequential stochastic control. The counterexample consists of a two-agent optimization 
problem with what has come to be known as a non-classical information pattern, in that it involves two agents acting in sequence, with the second agent having no knowledge of the information seen by the first agent. In the decades since \cite{witsenhausen1968counterexample}, a considerable literature
has sprung up analyzing control problems with non-classical information patterns \cite{yuksel2013stochastic}. Nevertheless, a complete analysis of the Witsenhausen's original counterexample is lacking, though considerable progress has been made in understanding the relation between 
optimal strategies and information patterns \cite{basar2008variations, uribe2014computing, gupta2015existence, kulkarni2015optimizer, jose2015linear, nayyar2013decentralized}. 

The goal of this paper is to contribute to the literature which attempts to explain why Witsenhausen's problem is difficult. Our starting point is the paper \cite{papadimitriou1986intractable}, which considered a discrete version of the Witsenhausen counterexample where all the
random variables and controls were restricted to be integers. This problem formulation can be obtained by quantizing the Witsenhausen problem and rescaling \cite{ho1980another, saldi2017finite}. Furthermore, the distribution
of the initial state of the system  and the noise were viewed as inputs; in Witsenhausen's original formulation, both of these were taken to be Gaussian. It was shown in \cite{papadimitriou1986intractable} that computation of the optimal strategy for this version of the Witsenhausen problem 
is NP-complete. 


While such NP-hardness results do not have any implications for Witsenhausen's original counterexample, in the generalized scenario where the initial state and noise have arbitrary distributions, they have a fairly powerful message. Indeed, let us consider what would count as a solution of the Witsenhausen problem in this more general scenario. Presumably, one would want a formula for the optimal strategy as a function of  the initial state and noise distributions. However, such a formula would be quite useless if it could not be 
evaluated efficiently. Thus at the very least there should exist an efficient algorithm for the computation of the optimal strategy, and it is exactly this that \cite{papadimitriou1986intractable} rules out. 

Our goal in this paper is to strengthen the results of \cite{papadimitriou1986intractable}. We seek to address the question of whether it is possible to find approximately optimal solutions to the Witsenhausen problem. It might initially seem that there are reasons to be hopeful. Indeed, the reduction in \cite{papadimitriou1986intractable} reduces 
the Witsenhausen problem to a 3D matching problem, and, although 3D matching is NP-hard, a $4/3+\epsilon$ approximation algorithm is available for any $\epsilon>0$ \cite{cygan2013improved}. Moreover, constant factor approximation results were derived in \cite{grover2009finite} for a different, but finite dimensional problem formulation, albeit 
with Gaussian noises. 

Unfortunately, our main result rules out the possibility of a favorable approximation with the discrete Witsenhausen problem with arbitrary initial state and noise distribution. We describe a family of examples, where the initial state is uniform over a set of $n$ integers, and the noise is uniform over a set of at most $n^2$ integers, and it is NP-hard to find a
strategy whose cost is upper bounded by $n^{2-\epsilon}$ times the cost of the optimal strategy, for any $\epsilon>0$. 

One might wonder if the multiplicative $n^{2-\epsilon}$ factor is the best one could do, i.e., if the problem might be even more difficult to approximate than that. In that direction, we show that if the initial distribution has support ${\cal X}$ and the noise distribution has support ${\cal Z}$, then it is always possible to approximate the optimal Witsenhausen strategy to within a multiplicative
factor of $|{\cal X}|^3 |{\cal Z}|^4$. Plugging in $|{\cal X}| = n$ and ${\cal Z} \leq n^2$ for the construction of the previous paragraph, we obtain that a multiplicative $n^{11}$ approximation is possible in that case. This shows that, while one might potentially improve the $n^{2-\epsilon}$-inapproximability result we described above, one cannot improve it too much.

The remainder of this paper is organized as follows. Section \ref{sec:background} contains technical background, including a formal definition of the Witsenhausen problem and the discretizations we  described above. Section \ref{sec:inapproximability} contains a proof of the $n^{2-\epsilon}$-inapproximability result while Section \ref{sec:approximability} contains a proof of the  $|{\cal X}|^3 |{\cal Z}|^4$ approximation result.

%
%

\section{Background\label{sec:background}} We begin with an informal description of the Witsenhausen problem. Two agents attempt to stabilize a system by bringing its state close to zero in two time steps. The first agent observes the initial state $X_0$, which we assume to be a random variable with a known distribution. The first agent applies the control $u_1$, so that the  state becomes $X_0 + u_1$. Now the second agent can only see a noisy version $X_0 + u_1 + Z$ of the  state, where $Z$ is some random variable with a known distribution. It applies a control $u_2$ which is therefore constrained to be a function of $X_0 + u_1 + Z$. The final cost depends only on the size of the control applied by the first agent as well as the final distance to the origin: 
\[ E \left[ u_1^2 + K (X_0 + u_1 + u_2)^2 \right], \] where $K>0$ is some constant. In particular, the control applied by the second agent is ``free.'' 


In the classical Witsenhausen counterexample \cite{witsenhausen1968counterexample}, it is assumed that the the initial state is $X_0 \sim N(0,\sigma_0^2)$ while the noise is $Z \sim N(0,1)$, but in this paper we will consider arbitrary distributions for $X_0,Z$. 
We will find it convenient to reformulate the problem in a way that makes the inherent constraints explicit as follows. Given independent random variables $X_0, Z$ we are looking for maps $T: \R \rightarrow \R$, $\delta: \R \rightarrow \R$ which minimize the cost function \begin{small}
\begin{equation} \label{eq:wcost} E \left[ (T(X_0)-X_0)^2 + K (T(X_0) + \delta(T(X_0) + Z))^2 \right]. \end{equation} \end{small} 
Moreover, we will denote this quantity as $\Phi(p_{X_0}, p_{Z}, T, \delta)$ and refer to it as the ``Witsenhausen cost.'' Furthermore, we will refer to $E[(T(X_0)-X_0)^2]$ as the ``first-stage'' or ``transportation cost,'' while $E \left[ [T(X_0)+\delta(T(X_0) + Z))]^2\right]$ will be referred to as the ``second-stage'' cost.

The discrete Witsenhausen problem, defined formally next, is simply the restriction of this problem to random variables and maps which take on integer values. For convenience, in the sequel we use $\Z$ to denote the set of integers. 

\medskip

\begin{definition} Let $X_0, Z$ be independent bounded random variables taking on integer values with probability mass functions $p_{X_0}, p_Z$. The Witsenhausen problem asks for maps $T: \Z \rightarrow \Z$ and $\delta: \Z \rightarrow \Z$ achieving the minimum in Eq. (\ref{eq:wcost}). 
We will use $\Phi^*(p_{X_0}, p_{Z})$ to refer to the optimal cost as a function of the problem parameters.\footnote{A common convention in the literature is to specify the distribution of $Y=T(X_0) + Z$ conditioned on $T(X_0)$, but since $X_0$ and $Z$ are independent here, it is easier to simply specify the distribution of $Z$.} 
\end{definition} 

\medskip

This problem was essentially introduced in \cite{ho1980another}. It is not hard to see that an optimal solution exists: we can restrict our attention to a finite set of maps $T,\delta$, as there is no need to consider maps which move some values in ${\cal X}$ too far. 
It is also standard that, given $T$, 
the corresponding $\delta$ can be found by solving a  least squares problem. 






Our goal in this paper is to prove the following theorem, which refines a result of \cite{papadimitriou1986intractable} that the discrete Witsenhausen problem is NP-hard.  We will adopt the convention of using 
${\cal X}$ to refer to the support of $X$ and ${\cal Z}$ to refer to the support of $Z$. 

\begin{theorem} \label{thm:imposs} Consider the discrete Witsenhausen problem restricted to problem instances where $|{\cal X}|=n$ and $|{\cal Z}| \leq n^2$. Unless $P=NP$, for any $\epsilon > 0$ there does not exist a polynomial-time 
algorithm which returns a number $\phi$ satisfying 
\[ \phi \leq n^{2-\epsilon} \Phi^*(p_{X_0},p_Z) \] such that $\phi = \Phi(p_{X_0}, p_{Z}, T, \delta)$ for some choice of maps $T,\delta$. \label{thm:innap} \end{theorem}

\medskip

We also show that the $n^{2-\epsilon}$ factor in the inapproximability result cannot be improved too much.

\begin{theorem} There is a polynomial-time algorithm which returns $T,\delta$ satisfying 
\[  \Phi(p_{X_0}, p_{Z}, T, \delta) \leq |{\cal X}|^3 |{\cal Z}|^4 \Phi^*(p_{X_0},p_Z) \] \label{thm:polybound}
\end{theorem} 

\medskip

Indeed, plugging in $|{\cal X}|=n, |{\cal Z}|=n^2$ into this last theorem, we obtain that in the setting described by Theorem \ref{thm:innap}, this provides an $n^{11}$ multiplicative approximation.

\section{Proof of Theorem \ref{thm:innap}\label{sec:inapproximability}}

We now turn to a sequence of lemmas whose culmination will be the proof of Theorem \ref{thm:innap}. Our starting point is a definition which later on will be key to the way we will definite the initial state $X_0$.

\medskip

\begin{definition} An integer set $S$ is called a Sidon set of order $p$ if all the sums 
\[ s_1 + s_2 + \cdots + s_p \] with 
$s_1, \ldots, s_p \in S$ and $s_1 \leq s_2 \leq \cdots \leq s_p$ are distinct.

\medskip

For example, $S=\{1,2,4\}$ is a Sidon set of order $2$ because the pairwise sums of elements from this set are distinct; but  
$S=\{1,2,3,4\}$ is not a Sidon set of order $2$ because $3+3=4+2$.
\end{definition}

It is well-known that Sidon sets of arbitrary order exist and can be easily constructed. We will need the following variation of this fact, which is very similar to a lemma from \cite{papadimitriou1986intractable}.

\medskip

\begin{lemma} There exists a Sidon set $S$ of order $4$ with $|S|=k$ satisfying $S \subset \{1,2,\ldots, 20 k^8\}$. Moreover, it is possible to construct $S$ in polynomial time in $k$. \label{lem:sidon}
\end{lemma}

\begin{proof} We prove this lemma by induction. When $k=1$, we can simply choose $S=\{1\}$. Now suppose we have a Sidon set  $S = \{s_1, s_2, \ldots, s_k\}$, with $s_i$ being distinct positive integers, and $\max_{i=1, \ldots, k} s_i \leq 20 k^8$. 
We construct a Sidon set of size $k+1$ by choosing a positive integer $s_{k+1} \leq 20(k+1)^8$ to add to $S$. 

To ensure this works, we need that 
\[ s_a + s_b + s_c + s_d \neq s_e + s_f + s_g + s_{k+1} \] for all possible choices $a \leq b \leq c \leq d, e \leq f \leq g$ from $a,b,c,d,e,f,g \in \{1, \ldots, k+1\}$. 
In other words, we need to have
\[ s_{k+1} \neq s_a + s_b + s_c + s_d - s_e - s_f - s_g. \] There are at most $(k+1)^7$ choices of $a,b,c,d,e,f,g$, and each inequality produces at most one value for $s_{k+1}$ to avoid. Indeed, it is possible for an inequality to produce no value for $s_{k+1}$ to avoid, for example 
if $s_{k+1}$ cancels from both sides. However, if this does not happen, then the corresponding choice of $a,b,c,d,e,f,g$ results in one value for $s_{k+1}$ to avoid.

It follows that this produces a Sidon set of order $4$ if we 
just place $s_{k+1}$ outside this set of at most $(k+1)^7$ values. 
We will place $s_{k+1}$ in the range $(20k^8, 20(k+1)^8]$, which we can always do by the pigeonhole principle, since  \begin{small}
\[ 20 (k+1)^8 - 20 k^8 >  20 \cdot 8  k^7 > (2k)^7 \geq  (k+1)^7. \]
\end{small}
Finally, we can construct $S$ via the above procedure which clearly takes polynomial time.

\end{proof}

The bounds of the lemma are rather loose and it is possible to improve them, but they suffice for our purposes. 

Our next step is discuss a variation on the notion of the chromatic number which we will need. Although the following discussion seems unrelated to Sidon sets, we will bring the two concepts together in our NP-hardness proof later on.

The chromatic number of a graph is the minimum number of colors needed to color the vertices so that no adjacent vertices share the same color; we will use $\kappa(G)$ to denote the chromatic number of the graph $G$. We will need to use a certain notion which we call the
$l_2$-chromatic number, which as far as we know has not been previously considered, and we motivate this notion with the following discussion. 

We may formulate the search for the chromatic number of the graph $G=(\{1, \ldots, n\},E)$ as minimizing 
\[ \Phi(\gamma) := \max_{i=1, \ldots, n} \gamma(i) - \min_{i=1, \ldots, n} \gamma(i), \] over all functions $\gamma: V \rightarrow \Z$ satisfying $\gamma(i) \neq \gamma(j)$ for all $(i,j) \in E$. Indeed, the objective $\Phi(\gamma)$ is precisely the number of colors needed minus one. 

The quantity $\Phi(\gamma)$ may be thought of as a measure of dispersion. This motivated the introduction of  the $l_2$-chromatic number, which uses a slightly different measure of dispersion: the variance of the distribution of $\gamma(i)$ about zero. 

\medskip

\begin{definition} Given an undirected graph $G=(\{1,\ldots,n\},E)$, the $l_2$-chromatic number asks for a function $\gamma: V \rightarrow \Z$ satisfying $\gamma(i) \neq \gamma(j)$ for all $(i,j) \in E$ and  minimizing 
\[ \gamma^* := \frac{1}{n} \sum_{i=1}^n \gamma^2(i). \] \label{def:l2chrom}
\end{definition} 

\medskip

We remark that the graph $G$ should not have any self-loops, for otherwise the constraint $\gamma(i) \neq \gamma(j)$ for all edges $(i,j)$ in $G$ is impossible to satisfy. 

We next use the concept of Sidon sets to give a way to construct a discrete Witsenhausen problem starting from a graph. The ensuring sequence of lemmas will show that computation of the $l_2$-chromatic number on that graph will then be equivalent 
to computation of the optimal Witsenhausen strategy.

\medskip

\begin{definition} Given a graph $G=(\{1,\ldots,n\},E)$ and an integer $B$, construct an instance of the discrete Witsenhausen as follows:
\begin{itemize} 
\item Let $\{y_1, \ldots, y_n\}$ be Sidon set of order $4$ with $n$ elements, and set $x_i = l y_i$, where $ l = 4 \left( \lceil n^{1.5} \rceil + 1\right)$. Generate $X_0$ to be uniform over $x_1, \ldots, x_n$. 
\item Generate $Z$ to be uniform over all the pairs $(x_i - x_j)/2, (i,j) \in E$.
\item Set $K$ to be any number strictly bigger than $n^5$.  
\end{itemize} \label{def:reduc}
\end{definition} 

\medskip

Observe that the graph $G$ enters the definition of the corresponding discrete Witsenhausen problem solely through the distribution of $Z$. Observe further that the support of $Z$ always symmetric about the origin since $(i,j) \in E$ whenever $(j,i) \in E$. Note also that the support of the random variable $X_0$, i.e., the set $\{x_1, \ldots, x_n\}$, is a Sidon set of order $4$ with $n$ elements
 (because it is obtained via scaling each element of a Sidon set by the same factor $l$).  
Finally, note that this construction may be performed in polynomial time in $n$ as a consequence of Lemma \ref{lem:sidon} which tells us that the set $\{y_1, \ldots, y_n\}$ may be constructed in polynomial time; that all remaining operations take polynomial time is obvious. 

The equivalence of $l_2$-chromatic number on the original graph and the cost of the optimal Witsenhausen strategy in this construction is established in the following two lemmas.

\medskip

\begin{lemma} Suppose $0 \leq B \leq n^2$. If the  discrete Witsenhausen problem constructed in Definition \ref{def:reduc} has a solution with cost at most $B$, then  the $l_2$-chromatic number of the graph $G$ is at most $B$. \label{lem:ch->dw}
\end{lemma} 

\begin{proof} Let $T,\delta$ be maps which achieve a cost of at most $B$ in the resulting Witsenhausen problem. We will define 
\[ \gamma(i) = T(x_i) -  x_i, \] and argue that this choice of $\gamma$ works. The key observation is that, if the discrete Witsenhausen problem constructed in this way has a cost at most $B$, then it has zero second-stage cost, i.e., we must have with probability one that
\begin{equation} \label{eq:ssc-zero} T(X_0) = \delta(T(X_0)+Z).
\end{equation}  This follows because of the way $K$ was chosen. Formally, observe that if Eq. (\ref{eq:ssc-zero}) fails with positive probability, then, because $X_0,Z$ were constructed to be uniform over ${\cal X}$ and ${\cal Z}$, it fails with probability at least 
$(1/|{\cal X}|) (1/|{\cal Z}|)  \geq 1/n^3$. Moreover, when Eq. (\ref{eq:ssc-zero}) fails, then because both the left-hand side and the right-hand side of this equation are integer, it follows they differ by at least one.  Thus, in that case the expectation of the Witsenhausen cost of Eq. (\ref{eq:wcost}) is at least 
$(1/n^3) K \cdot 1 > B$. This is a contradiction. We have thus shown that Eq. (\ref{eq:ssc-zero}) holds with probability one.

In particular, this means that for all possible  $x_i,x_j \in {\cal X}$ such that $T(x_i) \neq T(x_j)$, and all possible  $z_a,z_b \in {\cal Z}$, we must have 
\begin{equation} \label{eq:ssneq} T(x_i) + z_a  \neq T(x_j) + z_b. \end{equation} Indeed, if Eq. (\ref{eq:ssneq}) fails, then it is immediate that a zero second-stage cost cannot be obtained. 

We now claim that, due to the way $X_0$ was defined in Definition \ref{def:reduc}, we can conclude that actually $T(x_i) \neq T(x_j)$ for all pairs $i,j=1, \ldots, n$, so that the conclusion of the previous paragraph is actually applicable to all pairs $i,j$. Indeed, suppose $T(x_i) = T(x_j)$ for some pair $i,j$.
 Since $|x_i - x_j| > 4 n^{1.5}$, we have that either $|T(x_i) - x_i| > 2 n^{1.5} $ or $|T(x_j) - x_j| > 2 n^{1.5}$. Either one of these will imply the first-stage transportation cost is strictly bigger than $n^2$ and thus strictly bigger than $B$. 

Putting the last two paragraphs together, we have that for all realizations $x_i, x_j, \in {\cal X}, z_i,z_j \in {\cal Z}$, we have that 
\[ T(x_i) + z_a \neq T(x_j) + z_b. \] In particular, 
\[ T(x_i) - T(x_j) \neq z_b - z_a, \] or 
\[ T(x_i) - x_i - (T(x_j) - x_j) \neq z_b - z_a - x_i + x_j. \] 
But if $(i,j)$ is an edge in $G$, then the right-hand side of this equations equals zero when
\[ z_b = \frac{x_i-x_j}{2}, z_a = -z_b, \] and these are both in ${\cal Z}$.    So we conclude that if $i$ and $j$ are neighbors in $G$, then
\[ T(x_i) - x_i - (T(x_j) - x_j) \neq 0 \] or 
$\gamma(i) \neq \gamma(j)$. Thus $\gamma(i)$ satisfies the constraint in the definition of the $l_2$-chromatic number (i.e., Definition \ref{def:l2chrom}). 

Finally, we observe that
\[ \gamma^* \leq \frac{1}{n} \sum_{i=1}^n \gamma(i)^2 = \frac{1}{n} \sum_{i=1}^n (T(x_i) - x_i)^2, \] and because the second-stage cost is zero, this is equal equal to the expected Witsenhausen cost, which is at most $B$ by assumption. 
\end{proof} 

\medskip

Note that Lemma \ref{lem:ch->dw} did not use that the support of $X_0$ is a Sidon set. The next lemma, which is just the converse of Lemma \ref{lem:ch->dw}, will use this fact. 

\medskip

\begin{lemma} Suppose $1 \leq B \leq n^2$. If the $l_2$-chromatic number of $G$ is at most $B$, then the discrete Witsenhausen problem constructed in Definition \ref{def:reduc} has a solution of cost at most $B$. \label{lem:uniquenoise}
\end{lemma}

\medskip

Before we give a proof of this lemma, we require the following fact. 

\begin{lemma} $(x_i - x_j)/2 \in {\cal Z}$ if and only if $(i,j)$ is an edge in $G$. \label{lem:zsupport}
\end{lemma}

\begin{proof} One direction is one immediate from Definition \ref{def:reduc}. On the other hand, suppose $(x_i - x_j)/2 \in {\cal Z}$. This means there exist neighbors $a,b$ in $G$  such that 
\[ \frac{x_i - x_j}{2} = \frac{x_a - x_b}{2} \] or 
\[ x_i + x_i + x_b + x_b = x_a + x_a + x_j + x_j \] Since $x_a \neq x_b$ and $\{x_1, \ldots, x_n\}$ is a Sidon set of order four, this implies that $x_i = x_a, x_b = x_j$. Thus $i$ and $j$ are neighbors. 
\end{proof} 

\begin{proof}[Proof of Lemma \ref{lem:uniquenoise}] Paralleling the proof of Lemma \ref{lem:ch->dw}, we define 
\[ T(x_i) = x_i + \gamma(i), \] where $\gamma$ is the coloring that achieves $l_2$-chromatic number at most $B$. For integers $x' \notin \{x_1, \ldots, x_n\}$, we can define $T(x')$ arbitrarily, as it does not affect the Witsenhausen cost. 
We will show that, with this choice, the second-stage cost is zero. Once this is shown, the proof will be complete as the $l_2$-chromatic number $(1/n) \sum_i \gamma^2(i)$ is just the transportation cost. 

To argue that the second stage cost is zero, we proceed by contradiction. The second stage cost is not zero if and only if 
there exist $x_i,x_j \in {\cal X}, z_a,z_b \in {\cal Z}$ with $T(x_i) \neq T(x_j)$ such that
\begin{equation} \label{eq:ssum} T(x_i) + z_a = T(x_j) + z_b \end{equation} But, as in Lemma \ref{lem:ch->dw}, we cannot have $x_i \neq x_j$ with $T(x_i)=T(x_j)$; indeed, by the same argument as Lemma \ref{lem:ch->dw}, this implies that $|T(x_i) - x_i| > 2 n^{1.5}$, which now contradicts the fact that $\gamma$ achieves $l_2$-chromatic number at most $B \leq n^2$. So the second stage cost is zero if and only if 
there exist $x_i,x_j \in {\cal X}, x_i \neq x_j, z_a, z_b \in {\cal Z}$ such that Eq. (\ref{eq:ssum}) is satisfied. Now observe we can write Eq. (\ref{eq:ssum})
\[ x_i + \gamma(i) + z_a = x_j + \gamma(j) + z_b \] or 
\begin{equation} \label{eq:temp1} x_i + z_a - x_j - z_b = \gamma(j) - \gamma(i). \end{equation} Now the way $Z$ was constructed in Definition \ref{def:reduc} means that there exist neighbors $c,d$ and neighbors $e,f$ such that 
\[ z_a = \frac{x_c - x_d}{2}, ~~ z_b = \frac{x_e-x_f}{2}.\] Plugging this into Eq. (\ref{eq:temp1}) and doubling both sides, 
\[ 2 x_i + x_c - x_d - 2x_j - x_e + x_f = 2 \left[ \gamma(j) - \gamma(i) \right] \] 
or 
\begin{equation} \label{eq:sidon-contra}
 (x_i + x_i + x_c + x_f) - (x_j + x_j + x_d + x_e) = 2 (\gamma(j) - \gamma(i))
\end{equation}

Now we consider two possibilities both of which lead to a contradiction. The left-hand side of Eq. (\ref{eq:sidon-contra}) is either zero or nonzero. 

If it is zero, then since $x_i \neq x_j$, and  $\{x_1, \ldots, x_n\}$ being a Sidon set of order $4$, we must have 
\[ x_i = x_d = x_e ~~~ \mbox{ and } ~~~ x_j = x_c = x_f. \] But this means that  $(x_j - x_i)/2 = (x_c - x_d)/2 \in {\cal Z}$ so that by Lemma \ref{lem:uniquenoise} we have that $i$ and $j$ are neighbors. But since the left-hand side of  Eq. (\ref{eq:sidon-contra}) is zero, we have that $\gamma(j) = \gamma(i)$ for a pair of neighbors $i,j$, a contradiction. 

On the other hand, if the left-hand side of Eq. (\ref{eq:sidon-contra}) is nonzero, then, since every $x_i$ is a multiple of $l$ by construction (recall Definition \ref{def:reduc}), the same left-hand side must have absolute value at least $l$. It follows that 
\[ | \gamma(j) - \gamma(i)| \geq \frac{l}{2} > 2 n^{1.5}, \] where the strict inequality used the definition of $l$. Thus 
 at least one of $|\gamma(i)|, |\gamma(j)|$ is strictly bigger than $n^{1.5}$. But this contradicts that 
the $l_2$-chromatic number is at most $B \leq n^2$. This concludes the proof. 
\end{proof}

\medskip

We now turn to an analysis of the $l_2$-chromatic number.  We begin with a lemma which shows that the $l_2$-chromatic number is not very far from the ordinary chromatic number. Recall that we use the notation $\kappa(G)$ for the ordinary chromatic number of  $G$. 

\medskip

\begin{lemma} \[ \kappa^2 \geq  \gamma^*  \geq \frac{1}{n} \frac{(\kappa-2)^3}{12}  \] \label{lem:l2}
\end{lemma}

\medskip

\begin{proof} For the first inequality, simply consider taking $\gamma(i)$ to be the color of vertex $i$, represented by an integer in the set  $\{1, \ldots, \kappa\}$, using a coloring that minimizes the number of colors used. 

For the second inequality, consider the optimal $\gamma$ in the definition of $l_2$-coloring. Let us translate the  $\gamma$ so that the smallest interval $I$ containing its range is symmetric about the origin, i.e., it equals either $[-a,a]$ or $[-a,a+1]$.  Observe that every element in $I$ is used, i.e., every element in $I$ equals $\gamma(i)$ for some $i$, for else it would be possible to obtain a $\gamma$ with smaller $l_2$ chromatic number. This implies that \[ \gamma^* \geq 2 \frac{1}{n} (1^2 + \cdots + a^2) \geq \frac{2}{3} \frac{a^3}{n}. \] On the other hand, the chromatic number is at most $2(a+1)$. Thus 
\[ \kappa \leq 2a+2 \leq 2 ((3/2)n \gamma^*)^{1/3} + 2 \] or 
\[ (\kappa - 2)^3 \leq 12 n \gamma^*, \] which is a rearrangement of the second inequality. 
\end{proof} 

\medskip

\begin{lemma} Unless $P=NP$,  for any $\epsilon > 0$, there exists no polynomial time algorithm which, given an undirected graph on $n$ vertices, returns a number between $\gamma^*$ and $n^{2-\epsilon} \gamma^*$. \label{np-hard}
\end{lemma} 

\begin{proof} It is possible to define the notion of a {\em fractional chromatic number} of a graph $G$, denoted by $\chi_f(G)$. We avoid giving a definition here\footnote{The interested reader may look at \url{http://mathworld.wolfram.com/FractionalChromaticNumber.html}} because we only need to use the following two facts about it: 
\begin{itemize} \item In Theorem 1.2 of \cite{zuckerman2006linear}, it was shown that, for any $\epsilon > 0$, it is NP-hard to distinguish between graphs $G$ on $n$ vertices with fractional chromatic number of $n^{\epsilon}$ from graphs with fractional chromatic number of $n^{1-\epsilon}$.
\item In \cite{lovasz1975ratio}, it was shown that the fractional chromatic number is a logarithmic approximation to the chromatic number, i.e., \[ \frac{\kappa(G)}{1+\log n} \leq \chi_f(G) \leq \kappa(G),\] where, recall, $\kappa(G)$ is the ordinary chromatic number; for more details, see the discussion in Section 3.3 of \cite{feige1998zero}.
\end{itemize} 

As remarked in \cite{feige1998zero}, these two facts imply that it is NP-hard to distinguish between graphs of chromatic number $n^{\epsilon}(1+\log n)$ and graphs with chromatic number $n^{1-\epsilon}$. Applying Lemma \ref{lem:l2}, it follows that it is NP-hard to distinguish between graphs with $\gamma^* \leq n^{2 \epsilon} (1+\log n)^2$ and graphs with $\gamma^* \geq \frac{1}{12n} (n^{1-\epsilon} - 2 )^3$.  We conclude that, for any $\epsilon > 0$, it is NP-hard to approximate $\gamma^*$ within a multiplicative factor of  less than
\[ \frac{(n^{1-\epsilon}-2)^3}{12 n^{1+2\epsilon} (1 + \log n)^2}. \]  Because this quantity can be lower bounded by $n^{2 - O(\epsilon)}$, this completes the proof. 
\end{proof} 

Finally, we are now able to provide a proof of our main result. 

\begin{proof}[Proof of Theorem \ref{thm:innap}] Consider a graph $G$ with $l_2$-chromatic number of $B$. Since every vertex can be colored by a different color, we have that $B \leq n^2$. Consider the discrete Witsenhausen problem constructed in Definition \ref{def:reduc}: putting together Lemma \ref{lem:ch->dw} and Lemma \ref{lem:uniquenoise}, we obtain that its optimal solution solution has cost $B$. Now observing that by Lemma \ref{np-hard}, it is NP-hard to approximate $B$ to within a multiplicative factor of $n^{2-\epsilon}$ completes the proof.

\end{proof}

\section{Proof of Theorem \ref{thm:polybound} \label{sec:approximability}}

We now describe an algorithm for the Witsenhausen problem whose approximation ratio is polynomial in $|\cal X|$ and $|\cal Z|$. We begin with an informal discussion intended to motivate our approach.  Parallelling our arguments in the previous section, we'll adopt the convention of saying that $i$ and $j$ ``collide'' if 
\begin{equation} \label{eq:coll} T(x_i) + z_a = T(x_j) + z_b,
\end{equation} for some $z_a, z_b \in \cal Z$. 

Our approach is simple: we ``interpolate'' between the optimal solution when $K=0$ (which results in $T(x_i)=x_i$) and $K \rightarrow +\infty$ (which results in a $T$  that avoids any collisions) by fixing the $k$ elements in ${\cal X}$ with the highest probabilities, and moving all the other entries in ${\cal X}$ to avoid collisions. We do this for all $k=1, \ldots, n$ where $n=|{\cal X}|$ and choose the best result. 

We outline the approach in the algorithm box below,
where we use the convention that  $p_i$ is the probability of $X_0 = x_i$ and $$p_1 \geq p_2 \geq \cdots \geq p_n.$$

\begin{algorithm}
\caption{ \label{alg}}
\begin{algorithmic}[1]
\STATE Input: distributions of $X_0,Z$
\FOR  {$k=0, \ldots, n$} 
        \STATE Set $T^k$ be a map that map that fixes $x_1, \ldots, x_k$ and maps $x_{k+1}, \ldots, x_n$ to values ensuring there are no collisions except between $x_1, \ldots, x_k$. \label{step:tselect}
        \STATE Choose $\delta^k$ to be the optimal second-stage map given $T^k$.\label{step:deltaselect}
\ENDFOR
\STATE Choose the pair among $(T^k,\delta^k), k = 1, \ldots, n$ with lowest Witsenhausen cost. \label{step:last}
\end{algorithmic}
\end{algorithm}

It is easy to see that this is a polynomial-time algorithm. Indeed,  step \ref{step:last} can easily be done in polynomial time: the cost of each pair $T^i, \delta^i$ is a sum over $|\cal X| |\cal Z|$ values. Second, step \ref{step:deltaselect} can also be done without difficulty, since the selection of the best second-stage map given the transportation map is an ordinary least-squares estimation problem. The following lemma discusses how to do step \ref{step:tselect} and implicitly gives an upper bound on the transportation cost of the $T^i$ chosen in that step.

\smallskip 

\begin{lemma}  Step \ref{step:tselect} can be done in polynomial time with $|T^k(x_j) - x_j| \leq |{\cal X}| |{\cal Z}|^2$ for all $j$. \label{lem:tselect}
\end{lemma} 

\smallskip

\begin{proof} Starting with $j=k+1$, we sequntially set $T(x_j)$ to be the closest value to $x_j$ that does not yield a collision; when we have set $T(x_n)$, we are done. When we consider $x_m$, looking at Eq. (\ref{eq:coll}), we have to avoid \[ T(x_m) = T(x_j) + z_b - z_a, ~~~ j <m, ~ z_a, z_b \in \cal Z, \] which rules out at most $(m-1)|{\cal Z}|^2$ different values. It follows that we can always assign $T(x_m)$ so that $|T(x_m) - x_m| \leq |{\cal X}| |{\cal Z}|^2$. Moreover, each step of this procedure requires examining at most $|{\cal X}| |{\cal Z}|^2$ possibilities, and the number of steps is at most $|\cal X|$, so this procedure is polynomial time. 
\end{proof} 

We can now proceed to the proof of Theorem \ref{thm:polybound}. Our first step is to introduce some notation. We let $\Phi_1(p_{X_0},p_Z,T,\delta)$ to be the first-stage (transporation) cost when $X_0,Z,T,\delta$ are the random variables and maps in the discrete Witsenhausen problem. Likewise, we will use $\Phi_2(p_{X_0},p_Z,T,\delta)$ to denote the second-stage cost. Occasionally, we will omit to write the $\delta$ in this notation, and it should be
 understood that $\delta$ is then selected to be the optimal choice for the given $T$.

\begin{proof}[Proof of Theorem \ref{thm:polybound}] We claim that  Algorithm \ref{alg} with the selection procedure of Lemma \ref{lem:tselect} returns a solution with cost $|{\cal X}|^3 |{\cal Z}|^4 \Phi^*$ where $\Phi^*$ is the optimal Witsenhausen cost.

Indeed, consider the optimal strategy $T^*, \delta^*$.  Let $l$ be the smallest index such that $T^*(x_l) \neq x_l$ (we can assume such an index exists, because otherwise Algorithm \ref{alg} finds the optimal solution when $k=n$ and there is nothing to prove). The transport cost incurred by $T^*$ is at least $p_l$. 

Now consider the $(T^k, \delta^k)$ when $k=l$. The transport cost incurred by $T^l$ is upper bounded by $(|{\cal X}| p_l)(|{\cal X}| |{\cal Z}|^2)^2$ because the probability of not landing at a fixed point is at most  $|{\cal X}| p_l$, in which case one moves by at most $|{\cal X}| |{\cal Z}|^2$ as a consequence of Lemma \ref{lem:tselect}. Thus the transport cost incurred by $T^l$ is at most $p_l |{\cal X}|^3 |{\cal Z}|^4$. Putting the last two paragraphs together,  
\begin{equation} \label{eq:first-stage} \Phi_1 (p_{X_0}, p_Z, T^l, \delta^l) \leq |{\cal X}|^3 |{\cal Z}|^4 \Phi_1^*(p_{X_0}, p_Z) 
\end{equation}

We now consider the second-stage cost of $T^l, \delta^l$. By construction whenever one of $(x_{l+1}, \ldots, x_n)$ is generated, the second-stage cost is zero. Defining $p'$ to be the distribution proportional to $(p_1, p_2, \ldots, p_{l-1})$, this means that 
\begin{equation} \label{eq:ss1} \Phi_2(p_{X_0}, p_Z, T^l, \delta^l) = (p_1 + \cdots p_{l-1}) \Phi_2(p', p_Z, I). \end{equation} where we use $I$ for the identity map and we used that $T^l$ fixes $x_1, \ldots, x_{l-1}$. 

On the other hand, consider the second stage cost under $T^*,\delta^*$. Let  $A$ be the event that $X_0 \in \{ x_1, \ldots, x_{l-1}\}$. The second-stage cost cannot be increased if the first agent transmits to the second agent whether $A$ has occurred or not. Thus 
\begin{equation} \Phi_2(p_{X_0}, p_Z, T^*, \delta^*) \geq (p_1 + \cdots p_{l-1}) \Phi_2 (p', p_Z, I). \label{eq:ss2} \end{equation}

Finally, comparing Eq. (\ref{eq:ss1}) and Eq. (\ref{eq:ss2}) 
we obtain  $\Phi_2(p_{X_0}, p_Z, T^*, \delta^*) \geq \Phi_2(p_{X_0}, p_Z, T^l, \delta^l)$. Putting this together with  Eq. (\ref{eq:first-stage}) completes the proof.  
\end{proof}



\section{Acknowledgments} The author would like to thank Dr. S. Kopparty and Dr. A. Bhangale for suggesting the reduction between the $l_2$-chromatic number and the ordinary chromatic number. The author would 
also like to acknowledge Dr. M. Agarwal for multiple discussions of this problem. 


\bibliography{witsb}{}

\begin{thebibliography}{10}

\bibitem{papadimitriou1986intractable}
C.H. Papadimitriou and J.N. Tsitsiklis.
\newblock Intractable problems in control theory.
\newblock {\em SIAM journal on control and optimization}, 24(4):639--654, 1986.

\bibitem{witsenhausen1968counterexample}
Hans~S Witsenhausen.
\newblock A counterexample in stochastic optimum control.
\newblock {\em SIAM Journal on Control}, 6(1):131--147, 1968.

\bibitem{yuksel2013stochastic}
S.~Y{\"u}ksel and T.~Basar.
\newblock {\em Stochastic networked control systems}.
\newblock Birkhauser, 2013.

\bibitem{basar2008variations}
T.~Basar.
\newblock Variations on the theme of the witsenhausen counterexample.
\newblock In {\em Decision and Control, 2008. CDC 2008. 47th IEEE Conference
  on}, pages 1614--1619, 2008.

\bibitem{uribe2014computing}
C.~Uribe, T.~Keviczky, and J.~H. van Schuppen.
\newblock Computing optimal control laws for finite stochastic systems with
  non-classical information patterns.
\newblock In {\em American Control Conference (ACC), 2014}, pages 5742--5747.
  IEEE, 2014.

\bibitem{gupta2015existence}
A.~Gupta, S.~Yuksel, T.~Basar, and C.~Langbort.
\newblock On the existence of optimal policies for a class of static and
  sequential dynamic teams.
\newblock {\em SIAM Journal on Control and Optimization}, 53(3):1681--1712,
  2015.

\bibitem{kulkarni2015optimizer}
A.~Kulkarni and T.P. Coleman.
\newblock An optimizer's approach to stochastic control problems with
  nonclassical information structures.
\newblock {\em IEEE Transactions on Automatic Control}, 60(4):937--949, 2015.

\bibitem{jose2015linear}
S.T. Jose and A.A. Kulkarni.
\newblock A linear programming relaxation for stochastic control problems with
  non-classical information patterns.
\newblock In {\em Decision and Control (CDC), 2015 IEEE 54th Annual Conference
  on}, pages 5743--5748, 2015.

\bibitem{nayyar2013decentralized}
A.~Nayyar, A.~Mahajan, and D.~Teneketzis.
\newblock Decentralized stochastic control with partial history sharing: A
  common information approach.
\newblock {\em IEEE Transactions on Automatic Control}, 58(7):1644--1658, 2013.

\bibitem{ho1980another}
Y.~Ho and T.~Chang.
\newblock Another look at the nonclassical information structure problem.
\newblock {\em IEEE Transactions on Automatic Control}, 25(3):537--540, 1980.

\bibitem{saldi2017finite}
N.~Saldi, S.~Y{\"u}ksel, and T.~Linder.
\newblock Finite model approximations and asymptotic optimality of quantized
  policies in decentralized stochastic control.
\newblock {\em IEEE Transactions on Automatic Control}, 62(5):2360--2373, 2017.

\bibitem{cygan2013improved}
M.~Cygan.
\newblock Improved approximation for 3-dimensional matching via bounded
  pathwidth local search.
\newblock In {\em Foundations of Computer Science (FOCS), 2013 IEEE 54th Annual
  Symposium on}, pages 509--518, 2013.

\bibitem{grover2009finite}
P.~Grover, A.~Sahai, and S.Y. Park.
\newblock The finite-dimensional witsenhausen counterexample.
\newblock In {\em Modeling and Optimization in Mobile, Ad Hoc, and Wireless
  Networks, 2009. WiOPT 2009. 7th International Symposium on}, pages 1--10,
  2009.

\bibitem{zuckerman2006linear}
D.~Zuckerman.
\newblock Linear degree extractors and the inapproximability of max clique and
  chromatic number.
\newblock In {\em Proceedings of the thirty-eighth annual ACM symposium on
  Theory of computing}, pages 681--690, 2006.

\bibitem{lovasz1975ratio}
L.~Lov{\'a}sz.
\newblock On the ratio of optimal integral and fractional covers.
\newblock {\em Discrete mathematics}, 13(4):383--390, 1975.

\bibitem{feige1998zero}
U.~Feige and J.~Kilian.
\newblock Zero knowledge and the chromatic number.
\newblock {\em Journal of Computer and System Sciences}, 57(2):187--199, 1998.

\end{thebibliography}
\bibliographystyle{unsrt}

\end{document}